\documentclass[11pt]{amsart}
\usepackage[top=30truemm,bottom=30truemm,left=30truemm,right=30truemm]{geometry} 
\usepackage{amsmath, amssymb, array}
\usepackage[all]{xy}
\usepackage{ascmac}
\usepackage[dvips]{graphicx}
\usepackage{color}
\usepackage{amscd}
\usepackage[driverfallback=dvipdfm]{hyperref}
\usepackage{amsrefs}
\usepackage{cleveref}
\usepackage{shuffle}
\usepackage{mathrsfs}
\usepackage{listings}
\usepackage{comment}

\newfont{\cyr}{wncyr10 scaled\magstep0}

\newcommand{\Q}{\mathbb{Q}}
\newcommand{\Z}{\mathbb{Z}}

\DeclareMathOperator{\Tr}{Tr} 

\theoremstyle{plain}
 \newtheorem{theorem}{Theorem}[section]
 \crefname{theorem}{Theorem}{Theorems}
 
 \crefname{proposition}{Proposition}{Propositions}
 \newtheorem{lemma}[theorem]{Lemma}
 \crefname{lemma}{Lemma}{Lemmas}
 \newtheorem{corollary}[theorem]{Corollary}
 \crefname{corollary}{Corollary}{Corollaries}
 
 \crefname{conjecture}{Conjecture}{Conjectures}
 
 \crefname{hypothesis}{Hypothesis}{Hypotheses}
 
 \crefname{question}{Question}{Questions}
 
 \crefname{problem}{Problem}{Problems}

\theoremstyle{definition} 
 
 \crefname{definition}{Definition}{Definitions}
 
 \crefname{example}{Example}{Examples}
 \newtheorem{remark}[theorem]{Remark}
 \crefname{remark}{Remark}{Remarks}

\title{A note on the Diophantine equations $2ln^{2} = 1+q+ \cdots +q^{\alpha}$ and application to odd perfect numbers}
\author{Yoshinosuke Hirakawa}
\address[Yoshinosuke Hirakawa]{Department of Mathematics \\ Faculty of Science and Technology \\ Tokyo University of Science, 2641, Yamazaki, Noda, Chiba, Japan}
\email{hirakawa\_yoshinosuke@rs.tus.ac.jp \\ hirakawa@keio.jp}
	\thanks{This research was supported by KAKENHI 21K13779.}
	\subjclass[2010]{primary 11A41; 
	secondary
		11A05; 
		11D41; 
		11R11; 
		11R27; 
		11G30; 
		}
\keywords{perfect numbers; multiply perfect numbers; Diophantine equations; generalized Fermat equations}

\begin{document}


\maketitle

\begin{abstract}
Let $N$ be an odd perfect number. Then, Euler proved that there exist some integers $n, \alpha$ and a prime $q$ such that $N = n^{2}q^{\alpha}$, $q \nmid n$, and $q \equiv \alpha \equiv 1 \bmod 4$.
In this note, we prove that the ratio $\frac{\sigma(n^{2})}{q^{\alpha}}$ is neither a square nor a square times a single prime unless $\alpha = 1$.
It is a direct consequence of a certain property of the Diophantine equation $2ln^{2} = 1+q+ \cdots +q^{\alpha}$, where $l$ denotes one or a prime,
and its proof is based on the prime ideal factorization in the quadratic orders $\mathbb{Z}[\sqrt{1-q}]$ and the primitive solutions of generalized Fermat equations $x^{\beta}+y^{\beta} = 2z^{2}$.
We give also a slight generalization to odd multiply perfect numbers.
\end{abstract}



\section{Main result}

Let $N$ be a perfect number,
i.e., $\sigma(N) = 2N$,
where $\sigma(N)$ is the sum of the divisors of $N$.
Euler proved that
the even perfect numbers $N$ correspond bijectively to the Mersenne primes $q = 2^{p}-1$
as follows:
\[
	N
	= 2^{p-1}q.
\]
Historically, it was already known in Euclide's era that such $N$ is a perfect number.

In what follows,
we assume that $N$ is an odd perfect number.
Then, Euler proved that
there exist some integers $n, \alpha$ and a prime $q$ such that
\[
	N = n^{2}q^{\alpha},
	\quad q \nmid n,
	\quad \text{and}
	\quad q \equiv \alpha \equiv 1 \bmod 4.
\]
For its proof, see e.g.\ \cite[Theorem 2.2]{Chen-Luo}.
Since the quantity
\[
	\gcd(n^{2}, \sigma(n^{2}))
	= \gcd\left( \frac{\sigma(q^{\alpha})}{2} \cdot \frac{\sigma(n^{2})}{q^{\alpha}}, \frac{\sigma(n^{2})}{q^{\alpha}} \right)
	= \frac{\sigma(n^{2})}{q^{\alpha}},
\]
which is called the \emph{index} of $N$ in \cite{Gallardo},
is a counterpart of the quantity
\[
	\gcd(2^{p-1}, \sigma(2^{p-1}))
	= \frac{\sigma(2^{p-1})}{q}
	= 1
\]
for even $N$,
it is natural to ask whether $\frac{\sigma(n^{2})}{q^{\alpha}} = 1$ also for odd $N$ (e.g.\ \cite{Suryanarayana}).
The purpose of this note is to give an answer to this question in a more precise form as follows.

\begin{theorem} \label{2-perfect}
Suppose that $N = n^{2}q^{\alpha}$ is an odd perfect number as above.
\begin{enumerate}
\item
If $\alpha = 1$, then $\frac{\sigma(n^{2})}{q} > 1$.  

\item
If $\alpha > 1$, then $\frac{\sigma(n^{2})}{q^{\alpha}}$ is neither a square nor a square times a single prime.
\end{enumerate}
\end{theorem}


\begin{remark} \label{DHP1}
In \cite[Theorem 1]{Dandapat-Hunsucker-Pomerance},
Dandapat, Hunsucker, and Pomerance proved that
$\frac{\sigma(n^{2})}{q^{\alpha}} > 1$ for every $\alpha$
with some variants for multiply perfect numbers (see \cref{DHP}).
More recently, Gallardo \cite[Proposition 2]{Gallardo} showed that 
if $\alpha > 1$, then $\frac{\sigma(n^{2})}{q^{\alpha}}$ is not a square.
\cref{2-perfect} gives a refinement of these works at least partially.
\end{remark}

\section{Proof}

\subsection{The case of $\alpha = 1$}

\cref{2-perfect} for $\alpha = 1$ is well-known.
Let $\omega(n)$ be the number of distinct prime divisors of $n$
and $n = \prod_{i = 1}^{\omega(n)} p_{i}^{\alpha_{i}}$ with primes $p_{i}$ and integers $\alpha_{i} > 0$.
Then, the equality $q = \sigma(n^{2})$ implies that
\[
	q = \prod_{i = 1}^{\omega(n)} (1+p_{i}+\cdots+p_{i}^{2\alpha_{i}}),
\]
and hence $\omega(n) = 1$.
However, it contradicts the known lower bound e.g.\ $\omega(n) \geq 9$ due to Nielsen \cite[Theorem 3.8]{Nielsen}.

In fact,
we can deduce a huge lower bound for $\frac{\sigma(n^{2})}{q}$
from known results on the divisors of $n$;
for example, a single inequality $\omega(n) \geq 9$ implies that
\[
	\frac{\sigma(n^{2})}{q}
	\geq \sigma(3^{2}) \cdot \sigma(5^{2}) \cdot \sigma(7^{2}) \cdot \sigma(11^{2})
		\cdot \sigma(13^{2}) \cdot \sigma(17^{2}) \cdot \sigma(19^{2}) \cdot \sigma(23^{2})
	= 36163554870725919.
\]
For more discussion, see e.g.\ \cite{Nielsen}.

\subsection{The case of $\alpha > 1$}

Suppose that $\frac{\sigma(n^{2})}{q^{\alpha}} = ln_{0}^{2}$,
where $l$ is one or a prime and $n_{0}$ is an integer.
Then, since $2l(\frac{n}{ln_{0}})^{2} = \sigma(q^{\alpha})$,
\cref{2-perfect} is a consequence of the following property 
of the Diophantine equations $2ln^{2} = \sigma(q^{\alpha})$.

\begin{theorem} \label{main1}
There exist no triples $(\alpha, n, q)$ of integers such that
$\alpha > 1$, $q \equiv 1 \bmod 4$ is a prime, and $2n^{2} = \sigma(q^{\alpha})$.
\end{theorem}
\begin{theorem} \label{main2}
There exist no quartets $(\alpha, n, q, l)$ of integers such that
$\alpha > 3$, $q \equiv 1 \bmod 4$ is a prime, $l$ is a prime, and $2ln^{2} = \sigma(q^{\alpha})$.
\end{theorem}

First, we prove \cref{main1} by contradiction.

\begin{proof}[Proof of \cref{main1}]
Suppose that $2n^{2} = \sigma(q^{\alpha})$.
We may assume that $\alpha$ is odd and $\alpha \geq 3$.
In this case, the above equation can be rephrased as follows:
\[
	2(q-1)n^{2}
	= (q^{\frac{\alpha+1}{2}}-1)(q^{\frac{\alpha+1}{2}}+1).
\]
Since $q-1$ divides $q^{\frac{\alpha+1}{2}}-1$, $2$ divides $q^{\frac{\alpha+1}{2}}+1$, and 
\[
	\gcd(q^{\frac{\alpha+1}{2}}-1, q^{\frac{\alpha+1}{2}}+1)
	= \gcd(q^{\frac{\alpha+1}{2}}-1, 2)
	= 2,
\]
there exist some (coprime) divisors $n_{1}, n_{2}$ of $n$ such that 
\[
	\begin{cases}
	(q-1)n_{1}^{2} = q^{\frac{\alpha+1}{2}}-1, \\
	2n_{2}^{2} = q^{\frac{\alpha+1}{2}}+1.
	\end{cases}
\]
In particular, the former equation implies that
\[
	(1+n_{1}\sqrt{1-q})(1-n_{1}\sqrt{1-q})
	= q^{\frac{\alpha+1}{2}}
	= (1+\sqrt{1-q})^{\frac{\alpha+1}{2}}(1-\sqrt{1-q})^{\frac{\alpha+1}{2}}.
\]
Since $1 \pm n_{1}\sqrt{1-q}$ are coprime in the ring $\Z[\sqrt{1-q}]$
and $1 \pm \sqrt{1-q}$ are prime elements of $\Z[\sqrt{1-q}]$ above $q$,
the uniqueness of the prime ideal decomposition outside the prime ideals dividing $1-q$ (see e.g.\ \cite[\S7]{Cox})
implies that
there exists some non-negative integer $v$ such that
\[
	1+n_{1}\sqrt{1-q}
	= \epsilon(1+\sqrt{1-q})^{v}(1-\sqrt{1-q})^{\frac{\alpha+1}{2}-v},
\]
where $\epsilon \in \Z[\sqrt{1-q}]^{\times} = \{ \pm1 \}$.
Moreover,
if $0 < v < \frac{\alpha+1}{2}$,
then $1+n_{1}\sqrt{1-q}$ must be divisible by $q$,
a contradiction.
Therefore, $v = 0$ or $\frac{\alpha+1}{2}$.
In any case,
we obtain the following equality
\[
	2
	= \epsilon\Tr_{\Q(\sqrt{1-q})/\Q}((1 \pm \sqrt{1-q})^{\frac{\alpha+1}{2}})
	= 2\epsilon + 2\epsilon(1-q)\sum_{i = 1}^{\left[ \frac{\alpha+1}{4} \right]} \binom{\frac{\alpha+1}{2}}{2i} 1^{\frac{\alpha+1}{2}-2i} (1-q)^{i-1},
\]
which implies that
\[
	1 \equiv \epsilon \bmod q-1,
	\quad \text{i.e.,} \quad
	\epsilon = 1.
\]
Therefore,
by combining it with the assumption that $\alpha \geq 3$,
we obtain the following identity
\[
	0
	= 1 + \sum_{i = 2}^{\left[ \frac{\alpha+1}{4} \right]} \frac{\binom{\frac{\alpha+1}{2}}{2i}}{\binom{\frac{\alpha+1}{2}}{2}} (1-q)^{i-1}
	= 1 + \sum_{i = 2}^{\left[ \frac{\alpha+1}{4} \right]} \frac{\binom{\frac{\alpha-3}{2}}{2i-2}}{2i-1} \cdot \frac{(1-q)^{i-1}}{i}.
\]
However,
the right hand side is a $2$-adic unit whenever $q \equiv 1 \bmod 4$,
a contradiction.
\end{proof}

In the above proof,
we obtain the following by-product.

\begin{theorem} \label{by-product}
There exist no triples $(\beta, n, q)$ of integers such that
$\beta \geq 1$, $q \equiv 1 \bmod 4$ is a prime, and $n^{2} = \sigma(q^{\beta})$.
\end{theorem}

\begin{remark}
The Diophantine equation
\[
	n^{\alpha} = \frac{q^{\beta+1}-1}{q-1},
	\quad \text{in integers $(\alpha, \beta, n, q)$ such that $\alpha, \beta, \lvert n \rvert, \lvert q \rvert \geq 2$}
\]
is known as the Nagell-Ljunggren equation.
For  $\alpha = 2$,
Ljunggren \cite{Ljunggren} determined its integral solutions $(\beta, n, q)$ as follows:
\[
	(\beta, n, q) = (3, 20, 7), (4, 11, 3),
\]
which implies \cref{by-product} and hence \cref{main1}.
Moreover, there are only four known integral solutions
\[
	(\alpha, \beta, n, q) = (2, 3, 20, 7), (2, 4, 11, 3), (3, 2, 7, 18), (3, 2, 7, -19),
\]
which are conjectured to be the all of the solutions.
For more information, see e.g.\ \cite{Bennett-Levin,Bugeaud-Mignotte} and references therein.
The rational solutions $(n, q)$ are also studied in literature (see e.g.\ \cite{Ivorra_cyclotomic} for $\alpha = 2$ and $\beta = 10, 12, 16$),
but the complete solution seems to be unknown even if $\alpha = 2$.
\end{remark}

Next, we prove \cref{main2}.
The core of our proof is the following lemma whose proof is based on Chabauty method \cite{Bruin_thesis} and Frey curve method \cite{Ivorra_pp2}.

\begin{lemma} \label{ternary}
There exist no integers $n \geq 1, q\geq 2, \beta \geq 3$ such that
$2n^{2} = q^{\beta}+1$ and $3 \nmid \beta$.
\end{lemma}

\begin{proof}
It is sufficient to prove the statement under the condition that
$\beta = 4$ or $\beta$ is an odd prime.
The statement for $\beta = 4$ is a consequence of the fact that
the equation $2y^{2} = x^{4}+1$ has only four rational solutions $(x, y) = (\pm 1, \pm 1)$,
which one can check in a standard manner.
The statement for $\beta = 5$ is a consequence of \cite{Bruin_thesis},
which states that if $(x, y, z)$ is a primitive non-trivial triple of integers
\footnote{
	Here and after,
	``primitive'' means that $\gcd(x, y, z) = 1$
	and ``non-trivial'' means that $xyz \neq 0$.
	}
satisfies $x^{5}+y^{5} = 2z^{2}$, then $(x, y, z) = (3, -1, 11), (-1, 3, 11), (1, 1, 1)$.
\footnote{
	In fact,
	one can check that the integral solutions of $2n^{2} = q^{5}+1$
	are $(n, q) = (\pm1, 1), (0, -1)$ e.g.\ by applying the command \texttt{Chabauty} in Magma \cite{Magma}.
	}
Similarly, the statement for a prime $\beta \geq 7$ is a consequence of \cite{Ivorra_pp2},
which states that if $(x, y, z)$ is a primitive non-trivial triple of integers
satisfies $x^{\beta}+y^{\beta} = 2z^{2}$, then $(x, y, z) = (1, 1, 1)$.
This completes the proof. 
\end{proof}

\begin{remark} \label{smaller_beta}
If $\beta = 3$,
then the equation $2y^{2} = x^{\beta}+1$ 
has only five integral solutions $(x, y) = (-1, 0), (1, \pm1), (23, \pm78)$,
which one can check in a standard manner;
one can check it e.g.\ by applying the command \texttt{IntegralPoints} in Magma \cite{Magma}.
\end{remark}

Using \cref{ternary},
we can prove \cref{main2} by contradiction in a quite elementary manner. 

\begin{proof}[Proof of \cref{main2}]
By the same argument as \cref{main1},
there exist some (coprime) divisors $n_{1}, n_{2}$ of $n$ such that 
\[
	\begin{cases}
	(q-1)l^{\delta}n_{1}^{2} = q^{\beta}-1, \\
	2l^{1-\delta}n_{2}^{2} = q^{\beta}+1,
	\end{cases}
\]
where $\beta := \frac{\alpha+1}{2} \geq 3$ and $\delta \in \{ 0, 1 \}$.
In view of \cref{by-product},
we may assume that $\delta = 1$.
Moreover,
since $q > 1$,
\cref{ternary} implies that $3 \mid \beta$.
In what follows,
However,
since we assume that $q > 1$ and $q \equiv 1 \bmod 4$,
it contradicts \cref{smaller_beta}.
This completes the proof.
\end{proof}

\section{Application to multiply perfect numbers}

A positive integer $N$ is called $k$-perfect if $\sigma(N) = kN$.
The classical perfect numbers are $2$-perfect in this sense.
In \cite{Chen-Luo}, Chen and Luo generalized Euler's theorem on odd $2$-perfect numbers to odd $k$-perfect numbers with even $k$ as follows.
In what follows, $v_{2}(k)$ denotes the integer $v$ such that $2^{v} \mid k$ but $2^{v+1} \nmid k$.

\begin{theorem}[{cf.\ \cite[Theorem 2.2]{Chen-Luo}} \label{Chen-Luo}
\footnote{
	Note that the statement of \cite[Theorem 2.2]{Chen-Luo} is incorrect.
	Here, the author recovered a correct statement from its proof.
	}
]
Let $k \geq 2$ be an even integer and $N$ be an odd $k$-perfect number.
Then, there exist an integer $s$ such that $1 \leq s \leq v_{2}(k)$,
some integers $n, \alpha_{1}, \dots, \alpha_{s}$,
and distinct primes $q_{1}, \dots, q_{s}$ such that
\[
	N = n^{2}q_{1}^{\alpha_{1}} \cdots q_{s}^{\alpha_{s}}
	\quad \text{and}
	\quad q_{i} \nmid n.
\]
Moreover,
there exist some non-negative integers $a_{1}, \dots, a_{s}, b_{1}, \dots, b_{s}$ such that
\[
	v_{2}(k) - s = \sum_{i = 1}^{s} a_{i} + \sum_{i = 1}^{s} b_{i},
	\quad p_{i} \equiv 2^{a_{i}+1}-1 \bmod 2^{a_{i}+2},
	\quad \text{and}
	\quad \alpha_{i} \equiv 2^{b_{i}+1}-1 \bmod 2^{b_{i}+2}.
\]
\end{theorem}

\begin{corollary} \label{Chen-Luo2}
Let $k$ be an integer such that $k \equiv 2 \bmod 4$ and $N$ be an odd $k$-perfect number.
Then, there exist some integers $n, \alpha$ and a prime $q$ such that
\[
	N = n^{2}q^{\alpha},
	\quad q \nmid n,
	\quad q \equiv \alpha \equiv 1 \bmod 4.
\]
\end{corollary}

Thanks to \cref{Chen-Luo2},
we can prove the following by the same argument as \cref{2-perfect}.

\begin{theorem} \label{k-perfect}
Let us keep the notation and conditions in \cref{Chen-Luo2}.
\begin{enumerate}
\item
If $\alpha = 1$, then $\frac{\sigma(n^{2})}{q} > 1$. 

\item
If $\alpha > 1$ and $\frac{k}{2}$ is a square, then $\frac{\sigma(n^{2})}{q^{\alpha}}$ is neither a square nor a square times a single prime.

\item
If $\alpha > 1$ and $\frac{k}{2}$ is a square times a single prime, then $\frac{\sigma(n^{2})}{q^{\alpha}}$ is not a square.
\end{enumerate}
In any case, $\sigma(n^{2}) \neq q^{\alpha}$.
\end{theorem}

\begin{proof}
For $\alpha = 1$, we use the lower bound $\omega(n) \geq k^{2}-2$ due to McCarthy \cite[Corollary]{McCarthy}.
For $\alpha > 1$,
it is sufficient to note that
\[
	2 \cdot \frac{\frac{k}{2} \cdot n^{2}}{\frac{\sigma(n^{2})}{q^{\alpha}}} = \sigma(q^{\alpha}),
\]
which contradicts \cref{main1,main2}.
\end{proof}

Here, it should be mentioned that
the result \cite[Theorem 1]{Dandapat-Hunsucker-Pomerance} of Dandapat, Hunsucker, and Pomerance can be restated as follows.

\begin{theorem} [{\cite[Corollary]{Dandapat-Hunsucker-Pomerance}}] \label{DHP}
Let $N = mq^{\alpha}$ be a $k$-perfect number with an integer $m$ and a prime $q$ such that $\sigma(m) = q^{\alpha}$ and  $q \nmid m$.
Then, $N = 672$ or an even perfect number.
\end{theorem}

Thus, by combining \cref{DHP} and \cref{Chen-Luo2},
we see that $\frac{\sigma(n^{2})}{q^{\alpha}} \neq 1$ for every odd $k$-perfect number $N = n^{2}q^{\alpha}$ with $k \equiv 2 \bmod 4$.
In this view,
\cref{k-perfect} is a refinement of this result under the condition that $N$ is odd and $\frac{k}{2}$ is a square.

\appendix
\section*{Acknowledgement}
The author would like to thank Takato Suzuki for valuable discussion on perfect numbers.
The author sincerely appreciate helpful comments on a draft from Tomokazu Kashio, Hayato Matsumoto, Ryusei Suzuki, and Naoaki Takada.

\end{document}